%% file: main.tex
\pdfoutput=1
\documentclass[reqno, 12 pt]{amsart}

\input{Generic}

\usepackage[center]{caption}

\newcommand{\sfU}{\mathsf{U}}
\newcommand{\sfB}{\mathsf{B}}
\newcommand{\sfV}{\mathsf{V}}
\newcommand{\wth}{\omega_{\operatorname{Th}}}
\newcommand{\cextgamm}{\calC_{\ext, \gamma}}
\newcommand{\bextl}{B_{\ext}(\calL)}
\newcommand{\kerch}{Kerckhoff{ }}
\newcommand{\pigamma}{\Pi_\gamma}
\newcommand{\gap}{\,\,\,\,\,\,}

\newcommand{\gammacyl}{\langle \gamma \rangle}

\newcommand{\proj}{\operatorname{proj}}
\newcommand{\typ}{\operatorname{Typ}}

\newcommand{\cte}{\frac{\Lambda^2}{h \Vol(\calM_g)}}

\title{A conjugacy class counting in Teich\-m\"u\-ller\ space}
\author{Pouya Honaryar}
\date{June 2021}

\begin{document}

\begin{abstract}
    Let $\gamma$ be a pseudo-Anosov homeomorphism and $X$ an element of the \Teich space of a genus $g$ surface. In this paper, we find asymptotics for the number of \pa{} homeomorphisms that are conjugate to $\gamma$ and the axis of their action on \Teich space intersects the ball of radius $R$ centered at $X$, as $R$ tends to infinity.
\end{abstract}

\maketitle

\section{Introduction} \label{introduction}

\subsection{Statement of results} \label{statement}
In \cite{MargulisThesis}, Margulis obtained asymptotics for
the volume growth and orbit counting for balls of large radius, in the setting of manifolds with negative curvature.
Similar asymptotics were obtained in \cite{abem} for the \Teich space. To state their results, 
let $\calT_g$ be the \Teich space of $S$, a surface of genus $g$, and denote the mapping class group of $S$ by $\Gamma$. Given $X, Y \in \calT_g$, let $B(X, R)$ be the ball of radius $R$ centered at $X$, where the distance here is measured with respect to the \Teich metric. Denoting the orbit of $Y$ under the action of $\Gamma$ by $\Gamma \cdot Y$, Theorem 1.2 of \cite{abem} gives
$$  | \Gamma \cdot Y  \cap B(X, R)|  \sim \cte e^{hR} 
\gap \text{ as } \gap R \rightarrow \infty. $$
Here, $h = 6g-6$ is the entropy of the \Teich geodesic flow with respect to  Masur-Veech measure and $\Lambda$ is the Hubbard-Masur constant \cite{abem, dumas}. The term $\Vol(\calM_g)$ is the normalized volume of the moduli space $\calM_g$ as explained at the end of Section 2.2 of \cite{abem}.
The cardinality of a finite set $S$ is denoted by $|S|$ and $A(R)$ is said to be asymptotic to $B(R)$, written $A(R) \sim B(R)$, if $A(R)/B(R) \rightarrow 1$ as $R \rightarrow \infty$.
Theorem 1.3 of of the same paper gives the following asymptotics for the volume of $B(X, R)$:
$$ \Vol(B(X, R))  \sim \frac{\Lambda^2}{h} e^{hR} 
\gap \text{ as } \gap R \rightarrow \infty.$$

Now fix $\gamma \in \Gamma$ to be a pseudo-Anosov homeomorphism and let $\calL_\gamma$ be the axis of its action on \Teich space, namely, the unique geodesic that is kept fixed by this action.
The cyclic group generated by $\gamma$, denoted by $\gammacyl$, acts on $\calT_g$ 
properly discontinuously, hence we can form the quotient to be the cylinder 
$ C_\gamma = \gammacyl \backslash \calT_g$.
The elements of $C_\gamma$ are of the form $[Y] = \gammacyl.Y$ for $Y \in \calT_g$.
Since the action of $\gamma$ on $\calL_\gamma$ is by translation, the quotient $\bar{\calL}_\gamma  = \gammacyl \backslash \calL$ is a closed geodesic in $\calC_\gamma$. Define 
$$ B(\bar{\calL}_\gamma, R) = \{ [Y] \in \calC_\gamma : d([Y], \bar{\calL}_\gamma) \leq R \},$$
where the distance in $\calC_\gamma$ is the one induced by the \Teich distance on its cover $\calT_g$. Define the $\Gamma$--orbit of $[X] \in \calC_g$ to be 
$$\Gamma \cdot  [X] = \{ [gX]: g \in \Gamma \}.$$
The goal of this paper is to establish orbit counting and volume asymptotics, similar to the ones obtained in \cite{abem}, for $B(\bar{\calL}_\gamma, R)$ instead of $B(X, R)$. To state these results, we need a few definitions.
Define 
$$B_{\Ext}(\calL_\gamma)  = \{ \zeta \in \mf: \inf_{X \in \calL_\gamma} \ext(\zeta, X) \leq 1 \}.$$
The action of $\gammacyl$ on $B_{\ext}(\calL_\gamma)$ is proper and discontinuous, given we remove the two endpoints of ${\calL}_\gamma$ from the set, hence we can form the quotient to be 
$$\calC_{\gamma, \ext} = \gammacyl \backslash B_{\Ext}({\calL}_\gamma).$$ 
The Thurston measure on $\mf$ induces a measure on $\calC_{\gamma, \ext}$, which we denote by $\nu$ as well.
The orbit counting asymptotics is given by:

\begin{introthm} \label{Cylresult}
    Let $\gamma$ and $\bar{\calL}_\gamma$ be as above and $[X] \in \calC_\gamma$. Then as $R \to \infty$,
    \begin{align} \label{CylresultEq}
         | \Gamma \cdot [X] \cap B(\bar{\calL}_\gamma, R) | \sim \cte
    \nu (\calC_{\gamma, \ext}) e^{hR}.
    \end{align}
\end{introthm}

Note that $\Gamma \cdot [X] \cap B(\bar{\calL}_\gamma, R)$ is in one-to-one correspondence with the mapping class group translations $g \calL_\gamma$ of $\calL_\gamma$ that intersects $B(X, R)$. This in turn is in one-to-one correspondence with the conjugates of $\gamma$ whose axis intersects $B(X, R)$. Therefore, we have the following:

\begin{introcor} \label{conjcounting}
    Let $\gamma$ be a \pa{} homeomorphism and $X \in \calT_g$. Then as $R \to \infty$,
    $$ | \{ \gamma' \in \Gamma : \gamma' \text{ is conjugate to $\gamma$ and } \calL_{\gamma'} \cap B(X, R) \neq \varnothing \} | 
    \sim 
    \cte \nu (\calC_{\gamma, \ext}) e^{hR}.$$
\end{introcor}

The volume asymptotics is given by:
\begin{introthm} \label{VolumeResult}
    Let $\gamma$ and $\bar{\calL}_\gamma$ be as above. Then as $R \to \infty$, 
    $$ \Vol (B(\bar{\calL}_\gamma, R)) \sim \frac{\Lambda^2}{h}
    \nu(\calC_{\gamma, \ext} ) e^{hR}. $$
\end{introthm}

\subsection{Remarks and the relation to other works}
If $\Sigma$ is a surface of constant negative curvature $-1$ and $\Gamma$ its fundamental group, then Theorem 2.5 of \cite{EskinMullen} gives
$$|B(\bar{\calL}_\gamma, R) \cap \Gamma \cdot [x]| \sim \frac{\operatorname{Length}(\bar{\calL}_\gamma)}{\operatorname{Area}(\Sigma)} e^R
\gap \text{ as } \gap R \to \infty, $$ 
%
where the terms in the above expression are defined in the same way as before. A calculation in hyperbolic metric shows Theorem \ref{VolumeResult} in this setting, namely, as $R \rightarrow \infty$,
$$ \operatorname{Area} (B(\bar{\calL}_\gamma, R)) \sim \operatorname{Length}(\bar{\calL}_\gamma) e^R.$$

For $M$ a compact manifold of (variable) negative curvature, we can define $\Gamma, \, \gamma, \, \calL_\gamma$ and $\bar{\calL}_\gamma$ similarly. The asymptotics for $| \Gamma \cdot [x] \cap B(\bar{\calL}_\gamma, R)|$ can be obtained as a special case of \emph{common perpendicular counting}. 
To explain this, let $\calL'_\gamma$ and $x'$ be the images of $\calL_\gamma$ and $x$ under the covering map $\Pi \from \widetilde{M} \to M$. 
Then $ \Gamma \cdot [x] \cap B(\bar{\calL}_\gamma, R)$ is in one-to-one correspondence with $\operatorname{Perp}(x', \calL'_\gamma, R)$, the perpendiculars from $x'$ to $\calL'_\gamma$ of length less than $R$, 
where such a perpendicular is defined as a locally geodesic path that starts from $x'$ and arrives perpendicularly at $\calL'_\gamma$.
It follwos from Theorem 1 of \cite{CommonPerp} that for some constant $c_\gamma > 0$, 
$$ |\Gamma \cdot [x] \cap B(\bar{\calL}_\gamma, R)| \sim c_\gamma e^{\delta R} 
\gap \text{ as } R \to \infty.$$
Here, $\delta$ is the topological entropy of the geodesic flow on $T^1 M$, the unit tangent bundle of $M$. Moreover, 
under some additional conditions, an exponentially small error term is obtained for the above asymptotics. (See Theorem 3 of the same paper for the precise statement.)

As a final remark, let us mention that our methods for proving Theorem \ref{Cylresult} are quite flexible. 
In particular, Theorem \ref{Cylresult} can be proved for an arbitrary compact set $\calK_\gamma \subset \calC_\gamma$ replacing $\bar{\calL}_\gamma$, and the proof is word for word the same, given we change an $\epsilon$--net of $\bar{\calL}_\gamma$ by an $\epsilon$--cover of $\calK_\gamma$. (see Section \ref{ProofOutline}.) 
Volume asymptotics in this case can be obtained in the same way as we obtained Theorem \ref{VolumeResult} from Theorem \ref{Cylresult}.

\subsection{The outline of the proof} \label{ProofOutline}
Theorem \ref{VolumeResult}, proved at the end of Section \ref{PrelimMainThm}, follows easily from Theorem \ref{Cylresult} using an estimate obtained in Theorem 5.1 of \cite{abem} .
Hence the main task is to prove Theorem \ref{Cylresult}. Let $\gamma$ and $\calL$ be as in Section \ref{statement} and fix a point $P \in \calT_g$. By definition, $\Gamma \cdot  [P] \cap B(\bar{\calL}_\gamma, R)$ is in one to one correspondence with 
\begin{align} \label{CountingSet}
    \gammacyl \backslash (\Gamma \cdot P \cap B(\calL_\gamma, R) ) = \{ \gammacyl.X  : X \in \Gamma \cdot P \cap B(\calL_\gamma, R)  \},
\end{align}
where
$$B(\calL_\gamma, R) = \{ X \in \calT_g : d(X, \calL_\gamma) \leq R \}. $$
Note that taking the quotient by $\gammacyl$ in (\ref{CountingSet}) is justified since $\Gamma \cdot P \cap B(\calL_\gamma, R)$ is $\gammacyl$--invariant.

Fix a point $O \in \calL$ and for an $\epsilon > 0$, let $O = X_0, X_1, ..., X_N = \gamma O$ be an $\epsilon$--net in $[O, \gamma O]$, i.e., the geodesics connecting $X_i$ to $X_{i + 1}$ are disjoint for $0 \leq i < N$ and $\sup_{0 \leq i < N} d(X_i, X_{i + 1}) < \epsilon$.
%
Translating this net by the powers $\gamma$, we get a $\gamma$--invariant $\epsilon$--net $(..., X_{-1}, X_0, X_1, ...)$ of $\calL_\gamma$. Define the closest point map
$$ \calP \from \pmf \to \ZZ \gap by \gap \calP [\zeta] = i \, \text{ if } \, \ext(\zeta, X_i) = \inf_{j \in \ZZ} \ext(\zeta, X_j).$$
Setting $\calA_i = \calP^{-1} (i)$, we obtain a $\gamma$--invariant partition of $\pmf$. 
Let $S(X_i, \calA_i, R)$ be the sector of radius $R$ centered at $X_i$ and observing $\calA_i$, namely,
all the points $Y \in \calT_g$ such that $d(X_i, Y) \leq R$ and the geodesic connecting $X_i$ to $Y$ hits the boundary at an element of $\calA_i$. (see \ref{PrelimMainThm} for a precise definition.)

The main geometric idea of this paper is that $ S(X_i, \calA_i, R)$'s are almost disjoint and they almost cover all of $B(\calL_\gamma, R)$. (See the discussion just before Lemma \ref{SectorsRDisjoint}.) 
Since  $\gamma S(X_i, \calA_i, R) = S(X_{i + N}, \calA_{i + N}, R)$, 
$$ \sum_{i = 0}^{N-1} |\Gamma \cdot  P \cap S(X_i, \calA_i, R) |$$
gives a good approximation for $|\gammacyl \backslash (\Gamma \cdot P \cap B(\calL_\gamma, R) )|$. The asymptotics of $|\Gamma \cdot  X \cap S(X_i, \calA_i, R) |$ as $R \rightarrow \infty$ is given by \cite{abem}. Summing up these asymptotics as the $\epsilon$--net $(X_i)$ in $\calL_\gamma$ gets finer, namely $\epsilon \to 0$, we obtain the right hand side of (\ref{CylresultEq}).

To make these ideas work, we need to approximate each $\calA_i$ from inside and outside by open sets $\calU_i \subset \calA_i \subset \calV_i$ and squeeze the above sum between the corresponding sums for $\calU_i$ and $\calV_i$  replacing $\calA_i$. To prove that these lower and upper bounds both converge to the right-hand side of (\ref{CylresultEq}), we need the boundary of the partition $\{ \calA_i \}$ to have measure zero.
This is proved in Section \ref{EquidistNeglig} and the proof uses Theorem \ref{ExtDerResult}. (see Section \ref{derivativeintro}) The only result from Section \ref{EquidistNeglig} that is used in the rest of the paper is Proposition \ref{measure zero}.
Section \ref{BusemannAprx} is devoted to the statement and proof of Proposition \ref{CompTeichExt}, which is the main tool we use to compare extremal and \Teich lengths. 
In Section \ref{MainTheorem}, we carry out the sector approximation scheme that we mentioned earlier. Both of the facts that $S(X_i, \calA_i, R)$'s are almost disjoint and that they almost cover $B(\calL_\gamma, R)$ are applications of Proposition \ref{CompTeichExt}.

\subsection{A formula for the derivative of extremal length} \label{derivativeintro}
We end this introduction by stating a formula that we obtained in Section \ref{EquidistNeglig} in the course of proving Proposition \ref{measure zero}.
For $X \in \calT_g$ and $\zeta \in \mf$, denote the extremal length of $\zeta$ in $X$ by $\ext(\zeta, X)$. 
Fixing $\zeta$, we can consider $E_\zeta = \ext(\zeta, \centerdot)$ as a function from $\calT_g$ to $\RR$.
This function is differentiable and its derivative at $X \in \calT_g$, $d_X E_\zeta \from T_X (\calT_g) \to \RR$, is given by the Gardiner's formula \cite{Gardiner}
$$d_X E_\zeta (\mu) = 2 \Re \int_X \mu. \calV_X^{-1} (\zeta),$$
where $T_X (\calT_g)$ is the tangent space to $\calT_g$ at  $X$, $\mu \in T_X (\calT_g)$ is a Beltrami differential and the homeomorphism $\calV_X \from Q(X) \to \mf$ is defined by sending a quadratic differential to its vertical measured foliation.

If we fix $X \in \calT_g$ instead, we can define
$$E_X \from \mf \to \RR \gap \text{ by } \gap
E_X (\zeta) =  \ext(\zeta, X). $$
In order to compute the derivative of $E_X$ we need a differential structure on $\mf$. In general, the manifold $\mf$ equipped with train-track charts is only piecewise linear.
However, if $\zeta$ is generic, meaning that it does not have a leaf connecting any two singularities and all the singularities are simple, then $\mf$ is smooth at $\zeta$ (in a sense to be defined at the beginning of Section \ref{ProofOfNeglig}), hence the tangent space at this point to $\mf$, $T_\zeta \mf$, is defined. The derivative of $E_X$ at such a $\zeta$ is given by the following theorem: (for the precise statement see Theorem \ref{DerOfE_X}.)

\begin{introthm} \label{ExtDerResult}
    Fix $X \in \calT_g$ and let $\zeta \in \mf$ be a generic measured foliation. Then $E_X$ is smooth at $\zeta$ and there is an $\eta \in T_\zeta \mf$ such that 
    $$ d_\zeta E_X \from T_\zeta \mf \to \RR 
    \gap \text{ is given by } \gap
    d_\zeta E_X(\centerdot) = \wth(\eta, \centerdot),$$
    where $\wth$ stands for the Thurston symplectic form (see Section \ref{backgroundTT} for a definition). Moreover, $\eta$ can be completely described in certain train track coordinates around $\zeta$.
\end{introthm}

\subsection*{Notation}
For a set $A$ and subsets $A_\delta$ indexed by $\delta \in (0, s)$ for some $s > 0$, we say $A_\delta \uparrow A$ as $\delta \downarrow 0$ if the follwoing holds: $A_{\delta_2} \supseteq A_{\delta_1}$ for $\delta_2 < \delta_1$ and $\bigcup A_\delta = A$. Similarly, we say $B_\delta \downarrow B$ as $\delta \downarrow 0$ if $B_{\delta_2} \subseteq B_{\delta_1}$ for $\delta_2 < \delta_1$ and $\bigcap B_\delta = B$. Finally, for real numbers $a, b, c$ we write $a \simeq_c b$ if $|a - b| < c$.

\subsection*{Acknowledgements} I would like to thank my advisor, Kasra Rafi, for suggesting the problem and his constant support during the writing of this paper. 

\section{Background on Teich\-m\"u\-ller\ space} \label{background}

\subsection*{\Teich space}
Let $S$ be a compact surface of genus $g \geq 2$. We denote the \Teich space of $S$ by $\calT_g$. This is the space of equivalent classes of orientation preserving homeomorphism  $f \from S \to X$, where $X$ is a Riemann surface and $f \from S \to X$ is said to be equivalent to $g \from S \to Y$ if there exists a biholomorphism $h \from X \to Y$ such that $g$ is isotopic to $h \circ f$. We denote an element $ [f \from S \to X ]$ of $\calT_g$ by $X$ and keep the marking in the back of our mind.
The mapping class group (or modular group) of $S$ is denoted by $\Gamma$. This is the group of orientation preserving homeomorphisms of $S$ up to isotopy.
An element of mapping class group $[\gamma \from S \to S]$ acts on $[f \from S \to X] \in \calT_g$  by change of marking, namely $[\gamma].[f] = [f \circ \gamma^{-1}]$.  
Taking the quotient of $\calT_g$ by $\Gamma$ we obtain the moduli space $\calM_g = \Gamma \backslash \calT_g$.

\subsection*{Quadratic differentials}
For a Riemann surface $X$, let $Q(X)$ be the space of holomorphic quadratic differentials (or quadratic differentials for short) on $X$. 
For a $\phi \in Q(X)$, define the norm of $\phi$ to be
$$ |\phi| = \int_X |\phi(z)| |dz|^2 .$$
The union of $Q(X)$ for $X \in \calT_g$ forms the space of quadratic differentials, denoted by $\QT_g$. 
More precisely, $\qt_g$ is the space of equivalent classes 
$[f \from S \to (X, \phi)]$, where $f$ and $X$ are as before and $\phi \in Q(X)$.
We denote $[f]$ by $(X, \phi)$ or just $\phi$.
Sending $(X, \phi)$ to $X$ gives a projection map $\pi \from \QT_g \to \calT_g$. 
The principal domain $ \calP(1, .., 1) \subset \qt_g$ is defined to be quadratic differentials with only simple zeros.

A flat chart for $(X, \phi)$ is a holomorphic chart $\varphi \from U \subset \CC \to X$ on which the pullback of $\phi$ is $dz^2$. 
The change of coordinates between two flat charts is of the form $z \to \pm z + c$.  
For $\phi \in \qt_g$ and $A \in \SL_2 (\RR)$,  $A. \phi$ is defined as the unique element $\psi \in \qt_g$ such that the change of marking $\phi \rightarrow A. \psi$ is given by
multiplication by $A$ on the corresponding flat charts. With this definition, \Teich geodesic flow is given by
$$g_t = \begin{pmatrix}
  e^t & 0\\ 
  0 & e^{-t}
\end{pmatrix}.$$

The cotangent space at $X \in \calT_g$ is naturally identified with $Q(X)$, hence the norm on $Q(X)$ induces a Finsler norm on \Teich space. The resulting metric is called the \Teich metric, the distance between two points $X, Y \in \calT_g$ is denoted by $d(X, Y)$ and the geodesic connecting $X$ to $Y$ is shown by $[X, Y]$. For $X \in \calT_g$ and $\zeta \in \mf$, there exists a unique $\phi \in Q(X)$ such that $\phi$ has $\zeta$ as its vertical measured foliation (\cite{hm}). 
For such $X$ and $\zeta$, define
$$[X, \zeta) = \{ \pi (g_t .\phi),\, 0 \leq t \leq \infty \}.$$




\subsection*{Extremal length}
Let $\calV \from \mathcal{QT}_g \to \mf$ be the function that sends a quadratic differential to its vertical measured foliation. For $X \in \calT_g$, the restriction of $\calV$ to $Q(X)$,  $\calV_X \from Q(X) \to \mf$, is a homeomorphism, and the extremal length of $\zeta \in \mf$ at such an $X$ can be defined by
$$\Ext(\zeta, X) = | \calV_X^{-1} (\zeta)|.$$

The Busemann functions are defined by
\begin{align*}
    \beta( \zeta, X) = \frac{1}{2} \log \ext (\zeta, X);\\
    \beta (\zeta, X, Y) = \beta (\zeta, Y) - \beta(\zeta, X).
\end{align*}
Note that $\beta (\zeta, X, Y)$ only depends on $[\zeta]$, hence it can be denoted by $\beta( [\zeta], X, Y)$.
Kerschoff inequality states that
$$ \frac{\ext(\zeta, Y)}{\ext(\zeta, X)} \leq 
e^{2 d(X, Y)},$$
hence taking logarithms we obtain 
$$ \beta(\zeta, Y) \leq \beta(\zeta, X) + d(X, Y).$$
If we think of $\beta(\zeta, X)$ as the "length at infinity" of $[X, \zeta)$, the above can be thought of as the triangle inequality in $\bigtriangleup(\zeta, X, Y)$.

\section{equidistant measured foliations are negligible} \label{EquidistNeglig}

\subsection{Background on train tracks} \label{backgroundTT}

We define a train track to be an embedded $3$--regular graph in $S$ such that its vertices are locally modeled on Figure \ref{switch}. The vertices of this graph are called \emph{switches} and the edges are called \emph{branches} of the train track. 
In the same figure, $a$ is called an \emph{incoming} branch and $b, c$ are called \emph{outgoing} branches.
A branch is said to be \emph{large} if it is the outgoing branch for both of its endpoints. A \emph{splitting} along the large branch $e$ is shown in Figure \ref{flip}.
A train track $\tau$ is said to be \emph{complete} if all the components of $S - \tau$ are cusped triangles.

\begin{figure}[ht]
    \setlength{\unitlength}{0.008\linewidth}
    \begin{picture}(20, 20)
    \put(-10,-3){\includegraphics[width=40\unitlength]{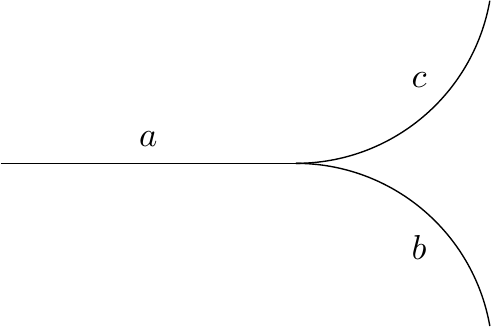}}
    \end{picture}
    \caption{A switch} 
    \label{switch} 
\end{figure}

A function $w$ from the set of branches of $\tau$ to $\RR$ is called a \emph{weight} if we have $w(a) = w(b) + w(c)$ for every switch as in Figure \ref{switch}. Let $W(\tau)$ be the set of all weights on $\tau$.
A weight $w \in W(\tau)$ is said to be \emph{positive} (or a \emph{measure} on the train track $\tau$), denoted by $w > 0$, if $w(a) > 0$ for all the branches $a$ of $\tau$. Denote the set of all measures on $\tau$ by $W^+(\tau)$.

Given $\mu \in W^+(\tau)$, we can foliate a rectangular neighborhood of $\tau$ according to $\mu$. Shrinking the components of the complement of this neighborhood, we get a measured foliation, denoted by $ \calF(\tau, \mu)$.
A measured foliation $\zeta$ is said to be \emph{carried} by $\tau$ if there exists a measure $\mu \in W^+ (\tau)$ such that $\zeta = \calF(\tau, \mu)$.

A measured foliation is called \emph{generic} if it has only simple singularities and does not have a leaf connecting any two of its singularities.
Let $\zeta \in \mf$ be generic and assume $\tau$ is a train track carrying $\zeta$, say, $\zeta = \calF(\tau, \mu)$ for some $\mu \in W^+(\tau)$. Since $\zeta$ is generic, $\tau$ should be complete, hence $W^+(\tau)$ is of maximal dimension $6g-6$ and
$$ \varphi_\tau \from W^+( \tau) \to \mf \gap \text{ defined by } \gap \mu_1 \mapsto \calF(\tau, \mu_1)$$
gives a chart around $\zeta$. 
If the train track $\tau'$ carries $\zeta$ as well, then the change of coordinates $\varphi_\tau^{-1} \circ \varphi_\tau$ is linear in a neighbourhood of $\mu$. This gives $\mf$ a linear structure at such a $\zeta$. (see the explanation after Proposition \ref{measure zero}.)

For a train track $\tau$, define the antisymmetric pairing
$\wth \from W(\tau) \times W(\tau) \rightarrow \RR$ by
\begin{align} \label{ThurstonForm}
    \wth(w_1, w_2) = 
    \frac{1}{2}
     \sum_{v} \det \begin{pmatrix}
        w_1(b_v) & w_1(c_v) \\
        w_2(b_v) & w_2(c_v)
     \end{pmatrix},
\end{align}
where the sum is over all the switches $v$ of $\tau$ and at each switch $v$, the incoming branch and outgoing branches are labeled by $a_v, b_v, c_v$ respectively, in such a way that $a_v b_v c_v$ is clockwise.
Since $W(\tau)$ is a vector space, $T_\mu W^+ (\tau)$ is naturally identified with $W(\tau)$ for every $\mu \in W^+ (\tau)$, hence (\ref{ThurstonForm}) gives an antisymmetric form on $W^+ (\tau)$, denoted by $\wth$ as well.
It can be proved that $\wth$ is invariant under the change of coordinates, hence it gives rise to an antisymmetric form on $\mf$, called the \emph{Thurston symplectic form}.


Let $( X, \phi) \in \qt_g$ and denote the zeros of $\phi$ on $X$ by $\Sigma$. A \emph{saddle triangulation} (or \emph{triangulation} for short) of $\phi$ is a triangulation of $X$ whose vertices belong to $\Sigma$ and the edges are straight lines in the flat metric induced by $\phi$. 
Fix such a triangulation $\Delta$ of $\phi$.
For a triangle $ABC \in \Delta$, a \emph{comparison triangle} is defined as a flat model of $ABC$, namely, this is a Euclidean triangle $A'B'C'$ together with a flat chart $\varphi \from A'B'C' \to ABC$ that sends $A'$ to $A$, $B'$ to $B$ and $C'$ to $C$. (By triangle here, we mean the union of edges and the interior.) 
Note that the comparison triangle is unique up to translation and reflection from the origin. 

For a triangulation $\Delta$, we can give the structure of a measured train track to the dual graph of $\Delta$ by defining the measure of an edge $e$, dual to the side $BC$ of a triangle $ABC \in \Delta$ to be $|\Re(\overrightarrow{B'C'})|$ where $A'B'C'$ is the corresponding comparison triangle. 
Note that if $\calV(\phi)$ is generic then $A'B'C'$ does not have a vertical side, hence the measure constructed above is indeed positive. 
The train track obtained in this way is called the train track \emph{adapted} to $\Delta$. Observe that, as shown in Figure \ref{flip}, a flip in the triangulation $\Delta$ corresponds to a splitting in the adapted train track and vice versa.

\begin{figure}[ht]
    \setlength{\unitlength}{0.015\linewidth}
    \begin{picture}(20, 20)
    \put(-10,0){\includegraphics[width=40\unitlength]{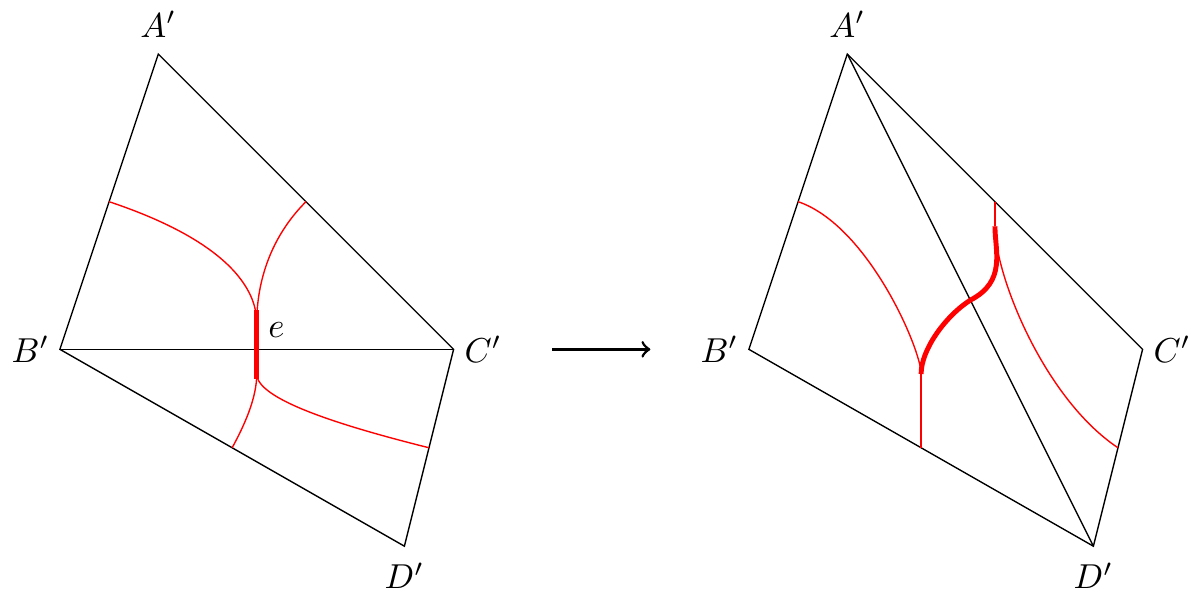}}
    \end{picture}
    \centering
    \caption{A splitting in the train track corresponds \newline to a flip in the triangluation} 
    \label{flip} 
\end{figure}

\subsection{$E(X, Y)$ has measure zero} \label{ProofOfNeglig}

For given $X, Y \in \calT_g$, define
$$E(X, Y) = \{ \zeta \in \MF : \Ext_X (\zeta) = \Ext_Y (\zeta) \}.$$

The goal of this section is to prove the following:

\begin{proposition} \label{measure zero} 
Let $X, Y \in \calT_g$ be distinct. Then $E(X, Y)$ is of Thurston measure zero.
\end{proposition}

A manifold $M$ with charts $\varphi_\alpha \from U_\alpha \to V_\alpha \subset M$ is said to be \emph{smooth} at $x \in M$ if the transition maps are smooth near $x$. 
More precisely, if for all indices $\alpha$ and $\beta$ such that $x \in V_\alpha \cap V_\beta$, $\varphi_{\alpha \beta} = \varphi_\beta ^{-1} \circ \varphi_{\alpha} \from U_{\alpha \beta} \to U_{\beta \alpha}$ is smooth on a neighbourhood of $\varphi_\alpha ^ {-1} x$ where $U_{\alpha \beta} $ is the domain of definition of $\varphi_{\alpha \beta}$.
We define a manifold to be \emph{linear} or \emph{analytic} at a point in a similar way.

If $M$ is smooth at $x$, the tangent space to $M$ at $x$, denoted by $T_x M$, can be defined in the usual way.
A function $f \from M \to \RR$ is said to be \emph{smooth} (\emph{analytic}, \emph{linear}) at a smooth (analytic, linear) point $x \in M$, if it is smooth (analytic, linear) in a neighborhood of $\varphi_\alpha^{-1}(x)$ for a chart $\varphi_\alpha$ that covers $x$. If $f$ is smooth at $x$,  the differential of $f$ at $x$, $d_x f \from T_x M \to \RR$, can be defined in the usual way.

For a given $X \in \calT_g$, define  
$ N \from Q(X) \to \RR$ by $N(\phi) = |\phi|$.
Since $Q(X)$ is a vector space, for every $\phi \in Q(X)$ we can identify the tangent space to $Q(X)$ at $\phi$, $T_\phi Q(X)$, with $Q(X)$. 
Define the following anti-symmetric pairing on $T_\phi Q(X)$:
$$ \omega_\phi (\psi_1, \psi_2) = \frac{1}{4} \Im \left(\int_X \frac{\psi_1 \bpsi_2}{|\phi|} \right). $$
Note that if $\phi \in \calP(1, ..., 1)$, then $\omega_\phi$ is defined for all $\psi_1, \psi_2$, but it's not necessarily so if $\phi$ has non-simple zeros.

\begin{lemma} \label{DerOfNorm}
    Given $X \in \calT_g$, let $\phi \in Q(X) \cap \calP(1, .., 1)$ and $\psi \in T_\phi Q(X) \simeq Q(X)$, then
    $$d_\phi N (\psi) = 4\omega_\phi (i\phi, \psi).$$
\end{lemma}

\begin{proof}
    The fact that $N$ is differentiable around $\phi$ follows from the proof of Theorem 5.3 of \cite{dumas}. With the notation introduced in that proof, we have
    $$N(\phi) = N^\epsilon_0 (\phi) + N^\epsilon_1(\phi),$$
    and it is proved that both $N^\epsilon_0$ and $N^\epsilon_1$ are smooth in a neighborhood of $\phi$. 
    The derivative of $N$ is computed in \cite{royden} Lemma 1.
\end{proof}

\textbf{Construction.}
Given $X \in \calT_g, \, \phi \in Q(X)$ and a triangulation $\Delta$ of $\phi$, let $(\tau, \mu)$ be the train track adapted to $\Delta$. For every $\psi \in Q(X)$, define $w(\Delta, \psi) \in W(\tau)$ as follows:
If the branch $e$ of the train track $\tau$ is dual to the side $BC$ of a triangle $ABC \in \Delta$, set
$$ w(\Delta, \psi)(e) = 
\frac{1}{2}
\Re \int_B^C  \frac{\psi}{\sqrt{\phi}},$$
where the integral is taken over the side $BC$ of the triangle $ABC$, and the sign for $\sqrt{\phi}$ is chosen so that $\Re \int_B^C \sqrt{\phi} > 0$.

\begin{lemma} \label{DerInTTCord}
    Let $X \in \calT_g$ and assume $\phi \in Q(X)$ is such that $\calV(\phi)$ is generic. If $\Delta$ is a triangulation of $\phi$ and $\tau$ is its adapted train track, then $\varphi_\tau^{-1} \circ \calV$ is defined and smooth (even real analytic) in a neighbourhood of $\phi$ and its derivative at $\phi$, $D_\phi (\varphi_\tau^{-1} \circ \calV) \from T_\phi Q(X) \to T_\mu W^+ (\tau)$, is given by
    $ \psi \mapsto w(\Delta, \psi)$.
\end{lemma}

\begin{proof}
    If $\phi_1 \in Q(X)$ is near $\phi$, we can choose a triangulation of $\phi_1$, denoted by $\Delta (\phi_1)$, that is close to $\Delta = \Delta(\phi)$. 
    Let $(\tau(\phi_1), \mu(\phi_1) )$ be the measured train track adapted to $\phi_1$.
    Since $\calV(\phi)$ is generic, $\tau(\phi_1)$ is the same as $\tau = \tau(\phi)$ up to isotopy.
    This gives us a map
    $$ W_\Delta \from U \to W^+ (\tau) \gap \text{ defined by } \gap W_\Delta(\phi_1) = \mu(\phi_1),$$
    where $U$ is a small neighbourhood of $\phi$. 
    Let $A$ and $B$ be two of the zeros of $\phi$ such that $AB$ is the side of a triangle in $\Delta$ and assume $e$ is the branch of $\tau$ that is dual to $AB$.  
    The zeros of $\phi_1$ vary as a complex analytic function of $\phi_1 \in Q(X)$, hence for $\phi_1$ close to $\phi$, $A_{\phi_1}$ and  $B_{\phi_1}$ can be chosen such that
    $$ W_\Delta(\phi_1) (e) =
    \int_{A_{\phi_1}}^{B_{\phi_1}} \Re \sqrt{\phi_1}.$$
    $\int_{A_{\phi_1}}^{B_{\phi_1}} \Re \sqrt{\phi_1}$ is called a period function and its derivative is given by Douady-Hubbard formula to be $w(\Delta, \psi)$ (\cite{D-H}).
\end{proof}

Recall that for $X \in \calT_g$,
$E_X \from \mf \to \RR$ is defined by $E_X(\zeta) = \ext(\zeta, X)$.

\begin{theorem} \label{DerOfE_X}
    Let $\zeta $ be a generic measured foliation and $X \in \calT_g$. Then $E_X$ is real analytic at $\zeta$ and its derivative, $d_\zeta E_X \from T_\zeta \mf \to \RR$, is given by
    \begin{align} \label{DerOfE_xEq1}
        d_\zeta E_X (\centerdot) = \wth (\eta, \centerdot)
    \end{align}
    for some $\eta \in T_\zeta \mf$ that depends on $X$ and $\zeta$.
    Moreover, assuming that $\Delta$ is a triangulation of $\phi = \calV_X^{-1} \zeta$ and $(\tau, \mu)$ its adapted train track, 
    then $\eta$ can be computed in the train track chart $\varphi_\tau$ to be $4 w (\Delta, i \phi)$, namely,
    \begin{align} \label{DerOfE_xEq2}
        \eta = D_\mu \varphi_\tau 
        \big(4 w (\Delta, i \phi) \big)
    \end{align}
    where $D_\mu \varphi_\tau \from W(\tau) \cong T_\mu W^+(\tau) \to T_\zeta \mf$ is the derivative of the chart $\varphi_\tau \from W^+(\tau) \to \mf$ at $\mu$.
\end{theorem}

\begin{proof}

    We start by taking derivatives of both sides of the identity
    $$ N(\phi_1) = E_X ( \calV(\phi_1))$$
    at $\phi$. Lemma \ref{DerOfNorm} gives the derivative of the left hand side to be  $\omega_\phi(4 i \phi, \centerdot)$. Using chain rule for the right hand side and the fact that $\wth$ is the push-forward of $\omega_\phi$ by $\calV_X$ (\cite{dumas} Theorem 5.8), we get
    $$d_\zeta E_X (\centerdot) = \wth (4 D\calV_\phi (i\phi), \centerdot).$$
    This is (\ref{DerOfE_xEq1}) for $\eta = 4 D\calV_\phi (i\phi)$. To obtain $D\calV_\phi (i\phi)$ in a train track chart, we should take derivative of both sides of 
    $$\calV_X(\phi_1) = \varphi_\tau \circ W_\Delta(\phi_1),$$
    where $W_\Delta$ is defined in the proof of Lemma \ref{DerInTTCord}.
    Chain rule and Lemma \ref{DerInTTCord} then gives (\ref{DerOfE_xEq2}).
\end{proof}

\begin{lemma} \label{DifferentPoints}
    Given $X, Y \in \calT_g$ and a generic measured foliation $\zeta$, let $\eta_1, \eta_2 \in \T_\zeta \mf$ be such that
    $ d_\zeta E_X = \wth( \eta_1, \centerdot)$  and   $d_\zeta E_Y = \wth( \eta_2, \centerdot)$.
    Then $\eta_1 = \eta_2$ implies $X = Y$.
\end{lemma}

\begin{proof}
    First let us describe $w(\Delta, i\phi)$ for an arbitrary triangulation $\Delta$ of a quadratic differential $\phi \in \qt_g$.
    Set $(\tau, \mu)$ to be the train track adapted to $\Delta$. Then by the definition of $w(\centerdot, \centerdot)$, 
    if a branch $e$ of $\tau$ is dual to the side $BC$ of a triangle $ABC \in \Delta$, then 
    $$w(\Delta, i\phi) (e) =- \Im(B'C'),$$ 
    where the comparison triangle $A'B'C'$ is chosen such that $\Re (\overrightarrow{B'C'}) > 0$.
    Note that $\phi$ here is uniquely determined by the data $(\tau, \mu, w(\Delta, i \phi) )$.
    
    Now to prove the lemma, let $\phi_1 \in Q(X)$ and $\phi_2 \in Q(Y)$ be such that $\calV(\phi_1) = \calV(\phi_2) = \zeta$.
    Choose an arbitrary triangulation $\Delta_j$ for $\phi_j$ and let $(\tau_j, \mu_j)$ be its adapted train track $(j = 1, 2)$. 
    Since $\zeta = \calF(\tau_1, \mu_1) = \calF(\tau_2, \mu_2)$, the two train tracks should have a common splitting $(\tau, \mu)$ (\cite{PH} Theorem 2.3.1). 
    Since every split along an edge corresponds to a flip in the dual triangulation, we obtain triangulations $\Delta'_1, \Delta'_2$ from $\Delta_1, \Delta_2$ such that they both have $(\tau, \mu)$ as their adapted train track.
    
    By Theorem \ref{DerOfE_X}, the derivative of $\varphi_\tau \from W^+(\tau) \to \mf$ identifies $\eta_j$ with $4w(\Delta'_j, i \phi_j) \in T_\mu W^+ (\tau)$ for $j = 1, 2$.
    Hence, the assumption $\eta_1 = \eta_2$ implies $w(\Delta'_1, i \phi_1) = w(\Delta'_2, i \phi_2)$. However, as mentioned before,  the data $(\tau, \mu, w(\Delta, i \phi) )$ uniquely determines $\phi$. This implies that $\phi_1 = \phi_2$, hence $X = Y$.
\end{proof}

\begin{proof}[Proof of Proposition \ref{measure zero}]
    Let $\calG$ be the subset of $\mf$ consisting of generic measured foliations. Then $\calG$ has full measure, i.e., $\nu (\mf \setminus \calG) = 0$.
    Define
    $$f \from \mf \to \RR \gap \text{ by } \gap f(\zeta) = E_X(\zeta) - E_Y(\zeta),$$
    hence $E(X, Y) = f^{-1} (0)$. 
    Let $\zeta \in  E(X, Y) \cap \calG$ be arbitrary, then according to Theorem \ref{DerOfE_X} there exist  $\eta_1, \eta_2$ such that
    $$d_\zeta E_X (\centerdot) = \wth (\eta_1, \centerdot)
    \gap \text{ and } \gap 
    d_\zeta E_Y (\centerdot) = \wth (\eta_2, \centerdot),$$
    hence the derivative $d_\zeta f \from T_\zeta \mf \to \RR$ is given by
    $$ d_\zeta f (\centerdot) = (
    \eta_1 - \eta_2, \centerdot).$$
    Since $X \neq Y$,  by Lemma \ref{DifferentPoints} we have $\eta_1 - \eta_2 \neq 0$.
    Non-degeneracy of the Thurston form (\cite{PH} Theorem 3.2.4) then implies that $d_\zeta f \neq 0$. Since $f$ is smooth at $\zeta$, $f(\zeta) = 0$ and $d_\zeta f \neq 0$, $f^{-1} (0)$ is locally a submanifold of codimension $1$ around $\zeta$.
    So there is a neighborhood $U_\zeta$ of $\zeta$ such that $E(X, Y) \cap U_\zeta$ is of measure zero.
    Now, covering $E(X, Y) \cap \calG$ by countably many $U_\zeta$'s we obtain $\nu(E(X, Y) \cap \calG) = 0$.
    Since $\calG$ has full measure, $\nu(E(X, Y)) = 0$.
\end{proof}

\section{Comparing \Teich and extremal lengths} \label{BusemannAprx}
\subsection{Projection to a thick geodesic} 
In this section, we state a few general facts about the geodesics that lie completely in the thick part, where by a geodesic we always mean a bi-infinite geodesic in the \Teich space, unless otherwise stated. Denoting the covering map from \Teich space to the moduli space by $\Pi \from \calT_g \to \calM_g$, for a subset $\calK \subset \calM_g$ we define $\widetilde{\calK}$ to be $\Pi^{-1} (\calK)$.
\begin{definition}
    Let $\calK \subset \calM_g$ be compact. A \Teich geodesic $\calG$ is said to be $\calK$--\emph{thick} if $\calG \subset \widetilde{\calK}$.
\end{definition}

Assume $\calG$ is $\calK$--thick for some compact $\calK \subset \calM_g$ and let $X \in \calT_g$ and $\zeta \in \mf$ be arbitrary. Define
\begin{align*}
    \proj_\calG X & = \{ Y \in \calG: d(X, Y) = d(X,\calG) \}; \\
    \proj_\calG [\zeta] & = \proj_\calG \zeta = \{Y \in\calG: \ext (\zeta, Y) = \ext (\zeta,\calG) \},
\end{align*}
where $d(X, \calG) = \inf \{ d(X, Y): Y \in \calG \}$ and $\ext(\zeta, \calG) = \inf \{ \ext(\zeta, Y): Y \in \calG \}$. Both $\diam (\proj_\calG X)$ and $\diam (\proj_\calG \zeta)$ are bounded by constants depending only on $\calK$, where $\diam$ stands for the diameter of a set. 
The boundedness of $\diam (\proj_\calG X)$ is a consequence of the contraction theorem of \cite{MinskyProjToThick} and the boundedness of $\diam (\proj_\calG \zeta)$ is also standard and follows, say, from Proposition \ref{extVshaped}.

There have been many analogies between the \Teich space, equipped with the \Teich metric, and a hyperbolic space. 
More specifically, we expect the \Teich metric to behave like a $\delta$--hyperbolic (Gromov hyperbolic) metric in the thick part. (see for example \cite{MinskyProjToThick}, \cite{CurveComplex}, \cite{hyperbolicinteich}.)
The following is an instance of this phenomenon:

\begin{theorem} \label{HypInTeich}
    \textbf{(\cite{hyperbolicinteich} Theorem 8.1)}
    Let $\calK \subset \calM_g$ be compact and $X, Y, Z \in \calT_g$. Then there are constants $C$ and $D$ only depending on $\calK$ such that the following holds:
    If $U, V \in [X, Y]$ are such that $[U, V] \subset \widetilde{\calK}$ and $d(U, V) > C$, then for every $W \in [U, V]$ we have
    $$\min\{ d(W, [Z, X]), d(W, [Z, Y]) \} < D.$$ 
\end{theorem}

Note that the above theorem remains true if any number of the vertices of the triangle $\bigtriangleup (X, Y, Z)$ belongs to the boundary of \Teich space. The following is a consequence of Theorem \ref{HypInTeich}:

\begin{proposition} \label{GeodCloseThick}
    Let $\calK \subset \calM_g$ be compact and $\calG$ be a $\calK$--thick geodesic.
    Then there exists a constant $C = C(\calK)$ such that for every $X \in \calT_g$, $Y \in \calG$ and  $H \in \proj_\calG X$, the geodesic connecting $X$ to $Y$ passes through $B(H, C)$, the ball of radius $C$ centered at $H$.
\end{proposition}

\begin{proof}
    Let $X, H, Y$ be as in the proposition. By Theorem \ref{HypInTeich} there exists $C' = C'(\calK)$ such that for every $Z \in [H,Y]$ there is $W \in [X, H] \cup [X, Y]$ such that $d(Z, W) < C'$. We claim that $C = 3C' + 1$ statisfies the statement. 
    If $d(H, Y) < 2C'+ 1$ then we are done, otherwise 
    let $Z \in [H, Y]$ be such that $d(H, Z) = 2C' + 1$. If the point $W$ given by Theorem \ref{HypInTeich} lies in $[X, H]$, then by triangle inequality in $\bigtriangleup (W,H,Z)$ we obtain $d(W, H) > C'+ 1$, hence
    $$ d(X, Z) \leq d(X, W) + d(W, Z) < d(X, H) - 1, $$
    which contradicts the choice of $H$. 
    This contradiction implies $W \in [X, Y]$. Triangle inequality in $\bigtriangleup(H,Z,W)$ then implies $d(H, W) < 3C' + 1$.
\end{proof}

\begin{corollary} \label{DistThickGeo}
    Let $\calK \subset \calM_g$ be compact and $\calG$ be a $\calK$--thick geodesic. 
    Then there is $C = C(\calK)$ such that for every $X \in \calT_g$, $H \in \proj_\calG X$ and $Y \in \calG$, we have
    $$ d(X, Y) \simeq_C d(X, H) + d(H, Y).$$
\end{corollary}

\begin{proof}
    Let $\calG$, $X$ and  $H$ be as above and let $C = C(\calK)$ be the constant given by Proposition \ref{GeodCloseThick}. If $[X, Y]$ intersects $B(H, C)$ at $Z$ then
    \begin{align*}
        d(X, Z) \simeq_C d(X, H) \text{    and     }
        d(Z, Y) \simeq_C d(H, Y).
    \end{align*}
    The statement follows from summing up these two estimates.
\end{proof}

Proposition \ref{GeodCloseThick} can be proved if $X \in \calT_g$ is replaced by a measured foliation $\zeta \in \mf$. The proof parallels the one given above, only instead of the triangle inequality for triangles with a vertex at infinity, we should use Kerschoff inequality. (see the discussion at the end of Section \ref{background}.) Corollary \ref{DistThickGeo} can be proved in this setting as well, hence we have the following:

\begin{proposition} \label{extVshaped}
    Let $\calK \subset \calM_g$ be compact and $\calG$ be a $\calK$--thick geodesic.
    Then there exists a constant $C = C(\calK)$ such that for every $\zeta \in \mf$, $H \in \proj_\calG \zeta$ and $Y \in \calG$, we have
    $$\beta(\zeta, Y) \simeq_C \beta(\zeta, H) + d(H, Y).$$
    Moreover, the geodesic connecting $\zeta$ to $Y$ passes through $B(H, C)$.
\end{proposition}

\subsection{Busemann approximation}
We make the following definition:

\begin{definition}
    Let $\calK \subset \calM_g$ be compact. The geodesic $[X, Y]$ connecting two points $X, Y \in \T_g$ is called $\calK $--\emph{typical} if it spends at least half of its time in $\widetilde{\calK}$
\end{definition}

We say that two geodesics $\calG_1 \from [0, a] \to \calT_g$ and $\calG_2 \from [0, b] \to \calT_g$, parametrized with respect to arc length, $D$--\emph{fellow travel}, if $|a - b| < D$ and for all $0 \leq t \leq \min\{a, b\}$ we have 
$$d(\calG_1(t), \calG_2(t)) < D.$$

\begin{theorem} \label{FellowTravel}
    \textbf{(\cite{hyperbolicinteich} Theorem 7.1)}
    Let $\calK \subset \calM_g$ be compact and $C > 0$. Then there exists a constant $D = D(\calK, C)$ such that the following holds: for every $X, Y \in \widetilde{\calK}$ and $\Bar{X}, \Bar{Y} \in \calT_g$ such that $d(X, \Bar{X})$ and $d(Y, \Bar{Y})$ are both less than $C$, the geodesics $[X, Y]$ and $[\Bar{X}, \Bar{Y}]$ $D$--\emph{fellow travel}.
\end{theorem}

Note that the conclusion of this theorem remains valid if $X = \Bar{X}$ belongs to the boundary of the \Teich space (\cite{hyperbolicinteich} Remark 7.2). In that case, we should allow $a = b = \infty$ in the definition of fellow traveling.

\begin{remark} \label{enlargecompact}
    It is a consequence of this theorem that for every compact set $\calK \subset \calM_g$ and real number $C > 0$, there exists an enlargement $\calK' \supset \calK$, depending only $\calK$ and $C$, such that if $X, Y, \Bar{X}, \Bar{Y}$ are as in the theorem and $[X, Y]$ is $\calK$--typical, then $[\Bar{X}, \Bar{Y}]$ is $\calK'$--typical. 
\end{remark}

Generically, two geodesic rays going to the same point in the boundary of \Teich space become exponentially close to each other. The precise statement is as follows:

\begin{theorem} \label{ExpConvRays}
    \textbf{(\cite{EMR} Corollary ??)}
    Let $\calK \subset \calM_g$ be compact and $C>0$. Then there are positive numbers $\alpha = \alpha(\calK)$ and $D = D(\calK, C)$ such that the following holds:
    If $X, Y \in \widetilde{\calK}$ and $\zeta \in \mf$ are such that $\ext(\zeta, X) = \ext(\zeta, Y)$; $d(X, Y) < C$; and $Z_1 \in [X, \zeta)$ and $Z_2 \in [Y, \zeta)$ are such that $d(X, Z_1) = d(Y, Z_2) = T$, then 
    $$d(Z_1, Z_2) < D e^{-\alpha T}.$$
\end{theorem}

We also need the following (\cite{MinskyProjToThick} Corollary 4.1):
\begin{lemma} \label{ProjNearBy}
    Let $\calK \subset \calM_g$ be compact and $\calG$ be a $\calK$--thick geodesic. Then there exists a constant $C = C(\calK)$ such that for every $X, Y \in \calT_g$ we have
    $$ \diam (\proj_\calG (X) \cup \proj_\calG(Y) ) < 
    d(X, Y) + C.$$
\end{lemma}

The next proposition is the main tool that we use to relate the extremal and \Teich lengths:


\begin{proposition} \label{CompTeichExt}
    Let $\calK \subset \calM_g$ be compact and $\calG$ be a $\calK$--thick geodesic. Then for every $\epsilon > 0$ there exists $C = C(\calK, \epsilon)$ such that the following holds:
    if $\zeta \in \mf$; $X, Y \in \calG$; $Z \in [X, \zeta)\cap \widetilde{\calK}$ and $H \in \proj (Z, \calG)$ are such that the goedesic $[Z, H]$ is $\calK$--typical and of length greater than $C$, then 
    $$d(Z, Y) - d(Z, X) \simeq_\epsilon 
    \beta(\zeta, Y) - \beta( \zeta, X) = \beta( [\zeta], X, Y).$$
\end{proposition}

\begin{proof}

    Let $H_\zeta \in \proj_\calG \zeta$ and $C = C(\calK)$ be the constant given by Proposition \ref{extVshaped}, so there is $X' \in (\zeta, X]$ such that $d(X', H_\zeta) < C$. 
    By Theorem \ref{FellowTravel} there is a constant $D$ depending on $C$ and $\calK$ such that $(\zeta, X']$ and  $(\zeta, H_\zeta]$ $D$--fellow travel; as a result, we can find $Z_\zeta \in (\zeta, H_\zeta$] with $d(Z, Z_\zeta) < D$. 
    Since $H_\zeta \in \proj_\calG Z_\zeta$, by Lemma \ref{ProjNearBy}, there exists $D'$ depending on $D$ and $\calK$ such that $d(H, H_\zeta) < D'$.
    Triangle inequality then implies $d(X', H) < D' + C$, hence $[Z, H]$ and $[Z, X']$ fellow travel, so $[Z, X']$ is $\calK'$--typical for some enlargement $\calK'$ of $\calK$ (Remark \ref{enlargecompact}) and since the constants $C, D, D'$ depend on $\calK$, the enlargement $\calK'$ only depends on $\calK$ as well.
    Applying Proposition \ref{extVshaped} once again gives $Y' \in (\zeta, Y]$ with $d(Y', H_\zeta) < C$, hence $d(X', Y') < 2C$.
    Using \kerch inequality we get
    $$|\beta(\zeta, X', Y')| \leq d(X', Y') < 2C,$$
    so by moving $Y'$ along $(\zeta, Y]$ by at most $2C$ we may obtain $Y''$ such that $\beta (\zeta, X') = \beta(\zeta, Y'')$.
 %
 %
%
    Let $Z'$ be the point obtained by flowing $Y''$ along $[Y'', \zeta)$ by time $T = d(X', Z)$. 
    By Theorem \ref{ExpConvRays} there exists $T_0 = T_0(\calK', C, \epsilon)$ such that if $T > T_0$ we have  $d(Z, Z') < \epsilon$. Hence if 
    $d(Z, H) > T_0 + D' + C$, we have
    \begin{align*}
        d(Z, Y) - d(Z, X) & \simeq_\epsilon d(Z', Y) - d(Z, X) = d(Y'', Y) - d(X', X) \\
        & = (\beta(\zeta, Y) - \beta(\zeta, Y'')) - (\beta(\zeta, X) - \beta(\zeta, X')) \\
        & = \beta(\zeta, Y) - \beta(\zeta, X).
    \end{align*}

\end{proof}

\section{Proofs} \label{MainTheorem}

\subsection{Preliminary discussion} \label{PrelimMainThm}
Fix a \pa{} homeomorphism $\gamma$ throughout this section  and let $\calL_\gamma$ be the axis of its action on \Teich space and $L = \tau(\gamma)$ its translation length.
Define $\calC_\gamma$ and $\Bar{\calL}_\gamma$ as in the introduction and denote the covering map from $\calT_g$ to $\calC_\gamma$ by $\pigamma$. 
To lighten the notation, we denote $\calL_\gamma$ and $\bar{\calL}_\gamma$ by $\calL$ and $\bar{\calL}$ respectively. 
Recall the following definitions:
\begin{align*}
    B(X, R) & = \{ Y \in \calT_g: d(X, Y) \leq R \};\\
    B(\calL, R) & = \{ Y \in \calT_g : d(Y, \calL) \leq R \};\\
    B_{\Ext}(\calL) & = \{ \zeta \in \mf: \Ext(\zeta, \calL) \leq 1 \};
\end{align*}
and define the following:
\begin{align*}
    \typ(\calL, \calK)  & = \{ Y \in \T_g : [Y, H] 
    \text{ is $\calK$--typical}
    \text{ for some } 
    H \in \proj_{\calL} Y \};\\
    \typ(X, \calK) & =  \{ Y \in \T_g : [X, Y] \text{ is $\calK$--typical} \}. 
\end{align*}
Given an open subset $\calU$ of $\pmf$, let
\begin{align*}
    S (X, \calU, R) & = \{ \pi (g_t \phi):  \phi \in Q(X),  [\calV(\phi)] \in \calU \text{ and } 0 \leq t\leq R \};\\
    S_{\Ext}(X, \calU) & = \{\zeta \in \mf : [\zeta] \in \calU  \text{ and } \ext(\zeta, X) \leq 1 \}.
\end{align*}

Let $X, P \in \calT_g$ and assume $\calU$ is an open subset of $\pmf$ . we need the following two facts from \cite{abem}:
\begin{itemize}
    \item As $R \rightarrow \infty$,
    \begin{align} \label{AbemForSector} 
        |\Gamma \cdot P \cap S(X, \calU, R) |   \sim \cte \nu (S_{\Ext} (X, \calU)) e^{hR}.
    \end{align}
    \item For every $\epsilon > 0$, there exists a compact set $\calK \subset \calM_g$ depending on $X, P, \epsilon$ such that
    \begin{align} \label{AbemSectorTypical}
    \limsup_{R \rightarrow \infty} e^{-hR} |\Gamma \cdot P \cap S(X, \calU, R) \setminus \typ(X, \calK)| < \epsilon.
    \end{align}
\end{itemize}

The first fact follows from \cite{abem} Theorem 2.9, since, with the notation of Proposition 2.1 of the same paper, we have 
\begin{align*}
    \int_\calU \lambda^-(q) ds_X (q) = 
    \int_\calU \frac{d(\delta^+_X)_* \Bar{\nu}}{ds_X} ds_X = (\delta_X^+)_* \Bar{\nu} (\calU) = 
    \nu \big( S_{\Ext} (X, \calU) \big),
\end{align*}
where the first equality uses part (i) of the same proposition.

The second fact follows from \cite{abem} Theorem 2.7 since (again, with the paper's notation) if $K' \subset \calP(1, ..., 1) \subset \qt^1_g$ is a compact subset of the principal domain of quadratic differentials with norm $1$, then $\calK = \pi (K')$ is compact as well.

Let $[\gamma^\pm]$ be the set containing the two elements of $\pmf$ that are fixed by $\gamma$. By Theorem 6.9 of \cite{Papadopoulos}, the action of $\gammacyl$ on 
$\pmf \setminus [\gamma^\pm]$ is a covering space action.
Define 
$$\calC_{\ext, \gamma} = \gammacyl \backslash B_{\ext} (\calL) \gap \text{ and let } \gap \Pi_{\ext, \gamma} \from B_{\ext} (\calL) \to \calC_{\ext, \gamma}$$
be the corresponding covering map.
Since $ \Gamma \cdot [P] \cap B(\bar{\calL}, R)$ is in one-to-one correspondence with 
$\gammacyl \backslash \big( \Gamma \cdot P \cap B(\calL, R) \big)$, Theorem \ref{Cylresult} is equivalent to the following:

\begin{theorem} \label{CylCount}
    Let $\gamma \in \mcg$ be pseudo-Anosov and $\calL = \calL_\gamma$ its axis. Then for a given  $P \in \calT_g$ we have 
    $$ | \gammacyl \backslash \big(\Gamma \cdot P \cap B(\calL, R)\big) | \sim \cte
    \nu (\cextgamm)
    e^{hR},$$
    as $R \rightarrow \infty$.
\end{theorem}

Theorem \ref{CylCount} is proved in Section \ref{proofofA}. Assuming this theorem, we now give a proof of Theorem \ref{VolumeResult}.

\begin{proof}[Proof of Theorem \ref{VolumeResult}]

    Fix a point $O \in \calL$ and let $Y \in \calT_g$ be arbitrary. 
    Pick a point $H_Y \in \proj_\calL(Y)$ and choose $k = k(Y) \in \ZZ$ in such a way that $d(O, \gamma^k H_Y) < L$. Define
    $$ h \from \gammacyl \backslash \big( \Gamma \cdot X \cap B(\calL, R) \big) \to \Gamma \cdot  X \cap B(O, R+L)$$
    by sending $\gammacyl.(gX)$ to $\gamma^{k{(gX)}} (gX)$
    (there might be more than one option for $H_{gX}$ and 
    $k{(gX)}$, then choose one.) Note that $h$ is an injection, hence
    \begin{equation} \label{ThmBineq1}
        |\gammacyl \backslash \big( \Gamma \cdot X \cap B(\calL, R) \big) | \leq 
    |\Gamma \cdot  X \cap B(O, R+L)|.
    \end{equation}

     By \cite{abem} Theorem 5.1, there exists a constant $C $, only depending on the base point $O$, such that for all $X \in \calT_g$ and $R>0$,
     \begin{align*} 
        |\Gamma \cdot  X \cap B(O, R)| < C e^{hR}.
     \end{align*}
    This, combined with (\ref{ThmBineq1}), implies that for $C' = C e^{hL}$ and every $X \in \calT_g$ we have 
    \begin{align} \label{ThmBineq2}
        |\gammacyl \backslash \big( \Gamma \cdot X \cap B(\calL, R) \big) | < C' e^{hR}.
    \end{align}
    
    Define the covering map
    $$ \Pi_{\gamma, \Gamma} \from \calC_\gamma \to \calM_g \gap \text{ by } \gap \gammacyl.X \mapsto \Gamma \cdot X.$$
    Since $\Pi_\gamma$ is a local diffeomorphism, we have
    \begin{align*} 
        \Vol (B(\bar{\calL}, R)) & = 
        \int_{\calM_g} |\Pi_{\gamma, \Gamma}^{-1}(X) \cap B(\bar{\calL}, R) | d\Vol(X) \\
        & = 
        \int_{\calM_g} |\gammacyl \backslash \big( \Gamma \cdot X \cap B(\calL, R) \big) | d\Vol(X).
    \end{align*}
    We multiply the left and right-hand side of this equation by $e^{-hR}$ and take the limit as $R \rightarrow \infty$. Since $C'$ in (\ref{ThmBineq2}) does not depend on $X$, we can apply Lebesgue's dominated convergence theorem to take the limit inside the integral. Theorem \ref{CylCount} then concludes the proof.
\end{proof}

\subsection{Concluding the proof} \label{proofofA}

The goal of this section is to give a proof of Theorem \ref{CylCount}.
Fix two points $P \in \calT_g$ and $O \in \calL$ for the rest of this section. 
To lighten the notation, we do not show the dependance of constants on $\calL, O, P$. So, for example, we write $C = C(\calK)$ instead of $C = C(\calK, \calL, O, P)$.
For $n, m \in \ZZ \cup \{ \pm \infty \}$, a sequence of points $\calX = (X_i)_{n \leq i \leq m} \subset \calL$ is called a \emph{net} if $(X_i, X_{i + 1})$'s are disjoint for $n \leq i < m$.
For such a net $\calX$, $\sup_{n \leq i < m} d(X_i, X_{i + 1})$ is called the \emph{mesh} of $\calX$ and for an $\epsilon > 0$, a net with mesh less than $\epsilon$ is called an $\epsilon$--\emph{net}.

Let $\epsilon > 0$ and take an $\epsilon$--net $O = X_0, ..., X_N = \gamma O$ in $[O, \gamma O]$.
Letting $X_{i + N} = \gamma X_i$, we obtain 
a $\gamma$--invariant $\epsilon$--net $\calX = (..., X_{-1}, X_0, X_1, ...)$ in $\calL$, and every $\gamma$--invariant $\epsilon$--net containing $O$ is obtained in this way. 
From now on, by an $\epsilon$--net $\calX = (X_i)_{i \in \ZZ}$ we always mean a $\gamma$--invariant $\epsilon$--net in $\calL$ with $X_0 = O$, and we show the number $N$ such that $X_N = \gamma O$ by $N = N(\calX)$. For such an $\epsilon$--net $\calX$, and for all $i \in \ZZ$, define 
\begin{align*}
    \calA_i(\calX)  = \{ \zeta \in \MF : \Ext(\zeta, X_i) \leq 1 \text{ and } \beta (\zeta,  X_i) = \inf_{j \in \ZZ} \beta (\zeta,  X_j) \}.
\end{align*}
\begin{lemma} \label{extconesqueeze}
    Let $\calX$ be an $\epsilon$--net, then
    $$\frac{1}{e^{\epsilon(6g - 6) }} \nu ( \cextgamm) \leq 
     \sum_{i = 0}^{N(\calX)-1} \nu(\calA_i(\calX)) 
     \leq \nu ( \calC_{\ext, \gamma}).$$
\end{lemma}

\begin{proof}
Let $N = N(\calX)$. It follows from the definition that the interior of $\calA_i(\calX)$'s, denoted by $\mathring{\calA}_i(\calX)$'s, are disjoint and $\calA_i(\calX)$'s are $\gamma$--\emph{equivariant}, meaning that $\gamma \calA_i(\calX) = \calA_{i + N} (\calX)$. These facts imply that
$\Pi_{\ext, \gamma} \from \mathring{\calA}_i(\calX) \to \cextgamm$ is an injection for $i \in \ZZ$, and $\Pi_{\ext, \gamma} (\mathring{\calA}_i(\calX)) \subset \calC_{\ext, \gamma}$ are disjoint for $0 \leq i < N$.
 Thus we have 
$$\sum_{i = 0}^{N-1} \nu(\mathring{\calA}_i(\calX)) 
\leq \nu ( \calC_{\ext, \gamma}).$$

By the definition of $E(\cdot, \cdot)$, given at the beginning of Section \ref{ProofOfNeglig}, $\calA_i(\calX) \cap \calA_j (\calX) \subset E(X_i, X_j)$, hence by proposition \ref{measure zero}
$$\partial \calA_i(\calX) = \bigcup_{j \neq i} \calA_i(\calX) \cap \calA_j(\calX) \subset \bigcup_{j \neq i} E(X_i, X_j) \implies \nu(\partial\calA_i (\calX)) = 0.$$ 
From this, we get
\begin{align*} 
    \sum_{i = 0}^{N-1} \nu(\mathring{\calA}_i(\calX)) 
    = \sum_{i = 0}^{N-1} \nu(\calA_i(\calX)) 
     \leq \nu ( \calC_{\ext, \gamma}).
\end{align*}
This proves the right hand inequality in the statement of the lemma. 

To prove the left hand inequality,
let $\zeta \in \bextl$ be arbitrary and assume $H_\zeta \in \proj_\calL \zeta$, so we have $\ext(\zeta, H_\zeta) \leq 1$. 
Since the mesh of $\calX$ is less than $\epsilon$, if $X_{i_0}$ is the element of the net that is closest to $H_\zeta$, by triangle inequality (\kerch inequality) in $\bigtriangleup (\zeta, H_\zeta, X_{i_0})$ we have
 $$\beta(\zeta, X_{i_0}) \leq \beta(\zeta, H_\zeta) + \beta(H_\zeta,  X_{i_0}) < \epsilon \implies 
 \inf_{i \in \ZZ} \beta(\zeta, X_i) < \epsilon \implies 
\inf_{i \in \ZZ} \ext(\zeta, X_i) < e^{2\epsilon}.$$
If the latter infimum is attained at $i = i_1$ then we have $\zeta/e^\epsilon \in \calA_{i_1}$, hence $\zeta \in e^{\epsilon} \calA_{i_1}$. As a result,
$$ B_{\Ext}(\calL) \subset \bigcup e^\epsilon \calA_i(\calX) \implies
\cextgamm \subset \bigcup  \Pi_{\ext, \gamma} (e^{\epsilon} \calA_i(\calX)),
$$
which implies 
\begin{align*}
    \frac{1}{e^{\epsilon(6g - 6) }} \nu ( \cextgamm) \leq \sum_{i = 0}^{N-1} \nu(\calA_i(\calX)).
\end{align*}

\end{proof}

Let
$\calX$ be an $\epsilon$--net and for $\delta > 0$, define the following subsets of $\MF$:
\begin{align*}
    \calU^\delta_i (\calX) & = \{ \zeta \in \MF : \Ext(\zeta, X_i) \leq 1 \text{ and } \beta (\zeta,  X_i) < \inf_{j \neq i} \beta (\zeta,  X_j) - \delta \}; \\
    \calV^\delta_i (\calX) & = \{ \zeta \in \MF : \Ext(\zeta, X_i) \leq 1 \text{ and } \beta (\zeta,  X_i) < \inf_{j \neq i} \beta (\zeta,  X_j) + \delta \}.
\end{align*}
Note that $\calU^\delta_i(\calX)$'s are open and $\gamma$--eqivariant, and the same is true for $\calV^\delta_i$'s. 
Moreover, $\calU^\delta_i (\calX)\subset \calA_i(\calX) \subset \calV^\delta_i$ and $\calU^\delta_i (\calX)\uparrow \mathring{\calA}_i(\calX)$ as $\delta \downarrow 0$; also, $\calV^\delta_i (\calX)\downarrow \calA_i(\calX)$ as $\delta \downarrow 0$. 
For a compact set $\calK \subset \calM_g$,
define the following subsets of $\Gamma \cdot  P$ associated to $\calU^\delta_i(\calX)$ and $\calV^\delta_i(\calX)$:
\begin{align*}
    \mathsf{U}^\delta_i (\calX, \calK) & = \Gamma \cdot  P \cap S \left(X_i, [\calU^\delta_i(\calX)] \right) \cap \typ (X_i, \calK);\\
    \mathsf{V}^\delta_i (\calX) & = \Gamma \cdot  P \cap S \left(X_i, [\calV^\delta_i (\calX)] \right),
\end{align*}
where for a subset $\calU \subset \mf$, we denote $\{ [\zeta]: \zeta \in \calU \} \subset \pmf$ by $[\calU]$.
As before, both $\mathsf{U}^\delta_i (\calX, \calK)$'s and $\mathsf{V}^\delta_i (\calX)$'s are $\gamma$--equivariant.
For the moment, let $N = N(\calX)$, $\sfU_i(R) = \sfU^{\delta}_i(\calX, \calK) \cap B(X_i, R)$ and 
$\sfV_i(R) = \sfV^{\delta}_i(\calX) \cap B(X_i, R)$.
As mentioned in \ref{ProofOutline}, the idea of the proof is to show that for $\calK$ large enough:

\begin{itemize}
    \item For every $i \in \ZZ$, $\pigamma \from \sfU_i(R) \to \Gamma \cdot [P] \cap B(\bar{\calL}, R)$ is almost an injection and 
    $\pigamma (\sfU_i(R) )$'s are more or less disjoint for $0 \leq i < N$(lemma \ref{SectorsRDisjoint}). Hence $\sum_{i = 0}^{N - 1} |\sfU_i(R)|$ gives a lower bound for $|\Gamma \cdot [P] \cap B(\bar{\calL}, R)|$

    \item The union of $\pigamma(\mathsf{V}_i(R+\epsilon))$'s for $0 \leq i < N$ cover almost all of $C(\bar{\calL}, R)$ (lemma \ref{CylMostlyTyp} + lemma \ref{SectorsCover}), hence $\sum_{i = 0}^{N - 1} |\sfV_i(R + \epsilon)|$ gives an upper bound for $|\Gamma \cdot [P] \cap B(\bar{\calL}, R)|$.
\end{itemize}

As $R \rightarrow \infty$, we can use (\ref{AbemForSector}) to count the points in each $\sfU_i (R)$ and $\sfV_i(R)$,
and if $\epsilon, \delta$ are small enough ($\delta$ moves to $0$ much faster than $\epsilon$), the upper and lower bounds obtained in this way are close to each other, and they approximate the right hand side of \ref{CylCount} from above and below.

\begin{lemma} \label{SectorsRDisjoint}
    Let $\calX =(X_i)_{i \in \ZZ}$ be an $\epsilon$--net.  Then, for every compact $\calK \subset \calM_g$ and $\delta>0$, there exists a constant $C = C (\epsilon, \calK, \delta)$ such that for $i \neq j$,
    $$ \left( \mathsf{U}^\delta_i (\calX,\calK) \setminus B(X_i, C) \right) \cap
    \left( \mathsf{U}^\delta_j (\calX, \calK) \setminus B(X_j, C) \right)
    = \varnothing $$
\end{lemma}

\begin{proof}
    Fix the net $\calX$ throughout the proof. Let $Y \in \mathsf{U}^\delta_i (\calX, \calK)$ and $H \in \proj_\calL Y$. We claim that there is a compact set $\calK'$, only depending on $\calK$, such that $[Y, H]$ is $\calK'$--typical.
    Let $\zeta \in \calU_i^\delta$ be such that $Y$ lies on $[X_i, \zeta)$. 
    Note that, since the net's mesh is less than $\epsilon$,
    $$\beta(\zeta, X_i) = \inf_{ k \in \ZZ} \beta (\zeta, X_k) \simeq_\epsilon \inf_{Y \in \calL} \beta(\zeta, Y).$$
    Because of the shape of $\beta(\zeta, X)$ as a function of $X \in \calL$, described in Proposition \ref{extVshaped}, the latter infimum should be attained near $X_i$, namely, there exists 
    $C_1  = C_1 (\calK)$ such that if $H_\zeta \in \proj_\calL \zeta$, then $d(X_i, H_\zeta) < C_1$.
    Thus, by Theorem \ref{FellowTravel}, $(\zeta, X_i]$ and $(\zeta, H_\zeta]$ $D$--fellow travel for some $D = D(\calK, C_1)$. 
    As a result, if $Z_\zeta \in (\zeta, H_\zeta]$ is such that $d(Y, X_i) = d(Z_\zeta , H_\zeta)$ then $d(Y, Z_\zeta) < D$. 
    Since $H_\zeta = \proj_\calL Z_\zeta$,  Lemma \ref{ProjNearBy} implies $d(H_\zeta, H) < D_1$ for some $D_1 = D_1(\calK, D)$. Triangle inequality then implies
    $d(X_i, H) < C_2 = C_1 + D_1$ 
    %
    and since $[Y,X_i]$ is $\calK$--typical, Remark \ref{enlargecompact} implies that $[Y, H]$ is $\calK'$--typical for an enlargement $\calK' \supset \calK$ that only depends on $\calK, C_2$. 

    This proves the claim.

    If we apply Proposition \ref{CompTeichExt} to $\calK'$ and $\delta/2$, we obtain $C_3 = C_3(\calK', \delta)$ such that if $Y$ and $H$ are as above, 
    $d(Y, \calL) = d(Y, H) > C_3$ and $[Y, H]$ is $\calK'$--typical, then
    $$ d(Y, X_i) - d(Y, X_j) \simeq_{\delta/2} 
    \beta(\zeta, X_i) - \beta(\zeta, X_j).$$
    Note that by triangle inequality in $\bigtriangleup(Y, H, X_i)$,
    $$ d(Y, H) \geq d(Y, X_i) - d(X_i, H) > d(Y, X_i) - C_2, $$
    so for $C = C_2 + C_3$,
    $$d(Y, X_i) > C \implies d(Y, H)> C_3.$$
    As a result, if $Y \in \mathsf{U}^\delta_i (\calX, \calK) \setminus B(X_i, C)$ we have 
    \begin{align*}
        & d(Y, X_i) - d(Y, X_j) \simeq_{\delta/2} d(\zeta, X_i) - d(\zeta, X_j) \\
        &\implies
        d(Y, X_i) - d(Y, X_j) < - \delta/2 < 0 \\
        &\implies d(Y, X_i) < d(Y, X_j). 
    \end{align*}

    The second impication is because $\zeta \in \calU^\delta_i(\calX)$ implies $d(\zeta, X_i) - d(\zeta, X_j) < -\delta$.
    If $Y \in \mathsf{U}^\delta_j(\calK) \setminus B(X_j, C)$ as well, then by changing the role of $X_i$ and $X_j$ in the argument above we obtain  $d(Y, X_j) < d(Y, X_i)$. This contradiction proves that the intersection mentioned in the lemma is empty.
    
\end{proof}

\begin{proof}[Proof of the lower bound for Theorem \ref{CylCount}]
  
For an $\epsilon$--net $\calX$, a compact set $\calK \subset \calM_g$ and $\delta >0$, let $C = C (\epsilon, \calK, \delta)$ be the constant given by Lemma \ref{SectorsRDisjoint}. For every $i \in \ZZ$, define 
$$ \sfU^\delta_i(\calX, \calK; C, R) = \mathsf{U}^\delta_i (\calX, \calK) \cap B(X_i, R) \setminus B(X_i, C).$$
We claim that
$$\Pi_\gamma \from \sfU^\delta_i(\calX, \calK; C, R)  \to 
\Gamma \cdot [P] \cap B(\bar{\calL}, R)$$ 
Is an injection. To prove this claim, assume $\Pi_\gamma (X) = \Pi_\gamma(Y)$ for $X, Y \in \sfU^\delta_i(\calX, \calK; C, R)$. 
This imlies that there exists $k\in \ZZ$ such that $\gamma^k X = Y$, so $Y$ belongs to the intersection of $\sfU^\delta_i(\calX, \calK; C, R) $ with
$ \gamma^k \sfU^\delta_{i}(\calX, \calK; C, R)  = \sfU^\delta_{i+ kN}(\calX,\calK; C, R)$. By the definition of $C$, we have $k = 0$, hence $X = Y$. This proves the claim.

A similar argument implies that $\Pi_\gamma \left( \sfU^\delta_i(\calX, \calK; C, R) \right)$ are disjoint for $0 \leq i < N(\calX)$, so
\begin{align*}
    \sum_{i = 0}^{N(\calX)-1} |\sfU^\delta_i(\calX, \calK; C, R) | \leq 
    |\Gamma \cdot [P] \cap B(\bar{\calL}, R)|.
\end{align*}

Multiplying both sides by $e^{-hR}$ and taking $\liminf$ as $R \rightarrow \infty$, we get 
\begin{align} \label{LowerBoundIneq}
   \sum_{i = 0}^{N(\calX)-1} s^\delta_i(\calX, \calK)
   & \leq \liminf_{R \rightarrow \infty} e^{-hR} |\Gamma \cdot [P] \cap B(\bar{\calL}, R)|,
\end{align}
where $s^\delta_i(\calX, \calK)$ is defined by
\begin{align*}
   s^\delta_i(\calX, \calK) & = \liminf_{R \rightarrow \infty} e^{-hR} |\sfU^\delta_i(\calK) \cap B(X_i, R)|
\end{align*}

Note that (\ref{LowerBoundIneq}) is valid for every $\epsilon$--net $\calX$, compact set $\calK \subset \calM_g$ and $\delta > 0$. 
Now, fixing $\calX$ and $\delta$, we let the compact sets $(\calK_n)_{n \in \NN} \subset \calM_g$ form an exhaustion of $\calM_g$. This means that $\calK_n \subset \mathring{\calK}_{n + 1}$ and $\calM_g = \bigcup_{n \in \NN} \calK_n$. Then, by (\ref{AbemForSector}) and (\ref{AbemSectorTypical}), 
$$s^\delta_i(\calX, \calK_n) \uparrow \cte \nu( \calU^\delta_i(\calX)) \gap \text{ as } \gap n \to \infty,
$$
so we get 
$$\cte \sum_{i = 0}^{N(\calX) - 1} \nu( \calU^\delta_i(\calX)) \leq \liminf_{R \rightarrow \infty} e^{-hR} |\Gamma \cdot [P] \cap B(\bar{\calL}, R)|.$$

Now, keeping the net $\calX$ fixed in the above expression and letting $\delta \downarrow 0$, we have
$\calU^\delta_i (\calX)\uparrow \mathring{\calA}_i (\calX)$, hence 
$\nu(\calU^\delta_i(\calX)) \uparrow \nu(\mathring{\calA}_i (\calX)) = \nu(\calA_i(\calX))$. Thus,
\begin{align*}
   \cte \sum_{i = 0}^{N(\calX)-1} \nu(\calA_i(\calX))
    \leq \liminf_{R \rightarrow \infty} e^{-hR} 
    |\Gamma \cdot [P] \cap B(\bar{\calL}, R)|.
\end{align*}

Finally, using Lemma \ref{extconesqueeze} and making the $\epsilon$--net $\calX$ finer, i.e., letting $\epsilon \rightarrow 0$, proves the lower bound
\begin{align*}
    \cte \nu(\cextgamm) \leq
    \liminf_{R \rightarrow \infty} e^{-hR} 
    |\Gamma \cdot [P] \cap B(\bar{\calL}, R)|.
\end{align*}

\end{proof}

For the upper bound we first prove:

\begin{lemma} \label{CylMostlyTyp}
    Let $\gamma$ and $\calL$  be as before. Then, for every $\kappa >0$, there exists a compact set $\calK \subset \calM_g$, only depending on $\kappa$, such that 
    \begin{align*}
        \limsup_{R \rightarrow \infty}
        e^{-hR} |\gammacyl \backslash \big( \Gamma . P \cap B(\calL, R) \setminus \typ(\calL, \calK) \big)| < \kappa.
    \end{align*}
\end{lemma}

\begin{proof}
    Let 
    $$ h \from \gammacyl \backslash \big( \Gamma \cdot P \cap B(\calL, R) \big) \to \Gamma \cdot  P \cap B(O, R+L)$$
    be the map defined in the proof of Theorem \ref{VolumeResult}, given after the statement of Theorem \ref{CylCount}. 
    By (\ref{AbemSectorTypical}), there exists $\calK' \subset \calM_g$, only depending on $O$ and $\kappa$, such that 
    \begin{align} \label{MostlyTypForBall}
        e^{-hR} |\Gamma \cdot P \cap B(O, R + L) \setminus \typ(O, \calK')| < \kappa
    \end{align}
    for $R$ large enough. 
    Let $[Y]  \in \gammacyl \backslash (\Gamma \cdot P \cap B(\calL, R)$ and recall that $h$ sends $[Y]$ to $ \gamma^kY$ such that $d(O, \gamma^k H) < L$ for some $H \in \proj_\calL Y$.
    Since $d(O, \gamma^k H ) < L$, by Remark \ref{enlargecompact}, there exists an enlargement $\calK \supset \calK'$, only depending on $\calK'$ and $L$, such that $[\gamma^k Y, O]$ $\calK'$--typical implies $[\gamma^k Y, \gamma^k H]$ is $\calK$--typical. 
    Thus, $h$ sends 
    $$\gammacyl \backslash \big( \Gamma \cdot P \cap B(\calL, R) \setminus \typ(\calL, \calK) \big) 
    \gap \text{ to } \gap \Gamma \cdot P \cap B(O, R+L) \setminus \typ(O, \calK').$$
    The injectivity of $h$ and (\ref{MostlyTypForBall}) then concludes the proof.
    
\end{proof}

\begin{lemma} \label{SectorsCover}
    Let the axis $\calL$ and the point $P \in \calT_g$ be as before, and assume $\calX =(X_i)_{i \in \ZZ}$ is an $\epsilon$--net.
    Then, for every compact set $\calK \subset \calM_g$ and $\delta > 0$, there exists $C = C(\epsilon, \calK, \delta)$ such that the following holds: if 
    $$Y \in \Gamma \cdot P \cap B(\calL, R) \cap \typ(\calL, \calK), \gap d(Y, \calL) > C,$$
    and $i_0 \in \ZZ$ is such that $d(Y, X_{i_0}) = \inf_{i \in \ZZ} d(Y, X_i),$
    then
    $$Y \in \sfV^\delta_{i_0}(\calX) \cap B(X_{i_0}, R + \epsilon).$$
\end{lemma}

\begin{proof}
Let
$$Y \in \Gamma \cdot P \cap B(\calL, R) \cap \typ(\calL, \calK).$$
Note that $\inf_{i \in \ZZ} d(Y, X_i)$ is attained for some $i \in \ZZ$ because of the shape of $d(Y, X)$ as a function of $X \in \calL$, described in Corollary \ref{DistThickGeo}.
Let the geodesic from $X_{i_0}$ to $Y$ hit the boundary at $[\zeta]$, i.e., $Y \in [X_{i_0}, \zeta)$. By Proposition \ref{CompTeichExt}, there exists $C = C(\calK, \delta)$ such that if $d(Y, \calL) > C$, then for all $i \in \ZZ$ we have
\begin{align*}
    & \beta(\zeta, X_i) - \beta(\zeta, X_{i_0}) \simeq_\delta d(Y, X_i) - d(Y, X_{i_0}) \\
    \implies & 
    \beta(\zeta, X_i) - \beta(\zeta, X_{i_0}) >
    d(Y, X_i) - d(Y, X_{i_0}) - \delta \geq - \delta \\
    \implies &
    \beta(\zeta, X_{i_0}) <  \beta(\zeta, X_i) + \delta\\
    \implies &
    \zeta \in \calV^\delta_{i_0} (\calX)
\end{align*}

Let $H \in \proj_\calL Y$ and note that since $Y \in B(\calL, R)$, we have $d(Y, H) \leq R$. Choose $X_{i_1}$ to be the point of $\calX$ that is closest to $H$. Since the mesh of $\calX$ is less than $\epsilon$, we have $d(H, X_{i_1}) < \epsilon$, hence by triangle inequality $d(Y, X_{i_1}) < R + \epsilon$. So 
$$ d(Y, X_{i_0}) = \inf_{j \in \ZZ} d(Y, X_i) < R+ \epsilon.$$
This proves the lemma.
 
\end{proof}

\begin{proof}[Proof of the upper bound for Theorem \ref{CylCount}]

For an $\epsilon$--net $\calX$, a compact set $\calK \subset \calM_g$ and $\delta >0$, let $C = C (\epsilon, \calK, \delta)$ be the constant given by Lemma \ref{SectorsCover}, and for $R > C$, write 
$$\Gamma \cdot P \cap B(\calL, R) = \sfB_1(\calK; C, R) \cup \sfB_2(\calK, R) \cup \sfB_3 (C),$$
where
\begin{align*}
    \sfB_1(\calK; C, R) & =  \Gamma \cdot P \cap B(\calL, R) \cap \typ(\calL, \calK) \setminus B(\calL, C) ;\\ 
    \sfB_2(\calK, R) & = \Gamma \cdot P \cap B(\calL, R) \setminus \typ(\calL, \calK);\\
    \sfB_3( C) & =  \Gamma \cdot P \cap B(\calL, C). 
\end{align*}
Note that $\sfB_1(\calK, C, R)$ ( $\sfB_2(\calK, R),\,\, \sfB_3(C)$ ) is $\gamma$--invariant, hence we can form the quotient $\overline{\sfB}_1(\calK, C, R) = \gammacyl \backslash \sfB_1(\calK, C, R)$ 
( $\overline{\sfB}_2 (\calK, R) = \gammacyl \backslash \sfB_2(\calK, R), \,\,\overline{\sfB}_3(C) =  \gammacyl \backslash \sfB_3(C)$ ) and obtain
\begin{align} \label{3sums}
    |\gammacyl \backslash \big(\Gamma \cdot P \cap B(\calL, R) \big)| \leq \overline{\sfB}_1(\calK, C, R) + \overline{\sfB}_2(\calK, R) + \overline{\sfB}_3(C).
\end{align}
Since $|\overline{\sfB}_3(C)|$ is a constant only depending on $C$, 
$$ \limsup_{R \rightarrow \infty} e^{-hR} |\overline{\sfB}_3(C)| \rightarrow 0.$$
To control $|\overline{\sfB}_2(\calK, R)|$, we define
%
$$ \kappa(\calK) = \limsup_{R \rightarrow \infty} e^{-hR} |\overline{\sfB}_2(\calK, R)|.$$
%
%
To find an upper bound for $\overline{\sfB}_1(\calK, C, R)$,
note that by Lemma \ref{SectorsCover},
\begin{align*}
    \sfB_1(\calK, C, R) & \subset \bigcup_{i \in \ZZ} \big( \sfV^\delta_i(\calX) \cap B( X_i, R + \epsilon) \big) \\
    \implies
    |\overline{\sfB}_1(\calK, C, R) | 
    & \leq \sum_{i = 0}^{N(\calX)-1} |\sfV^\delta_i(\calX) \cap B( X_i, R + \epsilon)|.
\end{align*}
  
Multiplying both sides of (\ref{3sums}) by $e^{-hR}$ and taking the $\limsup$ as $R \rightarrow \infty$ (while keeping $\calX$, $\calK$ and $\delta$ fixed), we obtain
\begin{align*}
    \limsup_{R \rightarrow \infty} e^{-hR} |\gammacyl \backslash \Gamma \cdot P \cap B(\calL, R)| & \leq
    \sum_{i = 0}^{N(\calX)-1} 
    \lim_{R \rightarrow \infty} e^{-hR} 
    |\sfV^\delta_i \cap B( X_i, R + \epsilon)| 
    + \kappa(\calK)\\
    & = \sum_{i = 0}^{N(\calX)-1} \cte e^{h \epsilon} \nu(\calV^\delta_i (\calX)) + \kappa(\calK),
\end{align*}
where the equality is by (\ref{AbemForSector}).
Note that the above is valid for every $\epsilon$--net $\calX$, compact set $\calK \subset \calM_g$ and $\delta > 0$. Keeping $\calX$ and $\delta$ fixed, we let $\calK_n \uparrow \calM_g$ to be an exhaustion of $\calM_g$. 
By Lemma \ref{CylMostlyTyp}, $\kappa(\calK_n) \downarrow 0$ as $n \to \infty$, thus
\begin{align*}
    \limsup_{R \rightarrow \infty} {e^{-hR} |\gammacyl \backslash \Gamma \cdot P \cap B(\calL, R)|}
    \leq \cte e^{h \epsilon} \sum_{i = 0}^{N(\calX)-1} \nu(\calV^\delta_i (\calX)).
\end{align*}
As in the proof of the lower bound, we first let $\delta \downarrow 0$ and use $\bigcap \calV^\delta_i (\calX)= \calA_i(\calX)$, then let $\epsilon \downarrow 0$ and use Lemma \ref{extconesqueeze} to obtain
\begin{align*}
    \limsup_{R \rightarrow \infty} {e^{-hR} |\gammacyl \backslash \Gamma \cdot P \cap B(\calL, R)|}
    \leq \cte \nu(\cextgamm).
\end{align*}

\end{proof}


\bibliographystyle{alpha}
\bibliography{references}

\end{document}

%% file: Generic.tex
\usepackage{amsthm} 
\usepackage{amsmath} 
\usepackage{amssymb} 
\usepackage{mathtools} 

\usepackage{graphicx}
\usepackage{xspace}
\usepackage[all]{xy}
\usepackage{pinlabel}
\usepackage{enumerate}
\usepackage{comment}

\usepackage[margin=1.3in]{geometry}
\usepackage{tikz}

\usetikzlibrary{cd, calc,through,backgrounds, arrows}
\usetikzlibrary{decorations.pathmorphing, decorations.pathreplacing,positioning,fit}

  \usepackage{hyperref}

\makeatletter

  \newcommand{\calA}{\mathcal{A}}
  
  \newcommand{\calC}{\mathcal{C}}

  \newcommand{\calF}{\mathcal{F}}
  \newcommand{\calG}{\mathcal{G}}

  \newcommand{\calK}{\mathcal{K}}
  \newcommand{\calL}{\mathcal{L}}
  \newcommand{\calM}{\mathcal{M}}

  \newcommand{\calP}{\mathcal{P}}

  \newcommand{\calT}{\mathcal{T}}
  \newcommand{\calU}{\mathcal{U}}
  \newcommand{\calV}{\mathcal{V}}
  
  \newcommand{\calX}{\mathcal{X}}

  \newcommand{\CC}{\mathbb{C}}

  \newcommand{\NN}{\mathbb{N}}

  \newcommand{\RR}{\mathbb{R}}

  \newcommand{\ZZ}{\mathbb{Z}}

  \newtheorem{theorem}{Theorem}[section]
  \newtheorem{proposition}[theorem]{Proposition}
  \newtheorem{corollary}[theorem]{Corollary}
  \newtheorem{lemma}[theorem]{Lemma}

  \newtheorem{introthm}{Theorem}
  \newtheorem{introcor}[introthm]{Corollary}

  \theoremstyle{definition}
  \newtheorem{definition}[theorem]{Definition}

  \newtheorem*{claim*}{Claim}

  \newtheorem*{question*}{Question}
  \newtheorem*{answer*}{Answer}
  \newtheorem*{application*}{Application}

  \theoremstyle{remark}
  \newtheorem{remark}[theorem]{Remark}
  \newtheorem*{remark*}{Remark}
  
  \DeclareMathOperator{\Ext}{Ext}

  \DeclareMathOperator{\Vol}{Vol}

  \DeclareMathOperator{\ext}{Ext}

  \newcommand{\SL}{\ensuremath{\operatorname{SL}}\xspace} 
   
  \newcommand{\T}{\ensuremath{\mathcal{T}}\xspace} 
  \newcommand{\MF}{\ensuremath{\mathcal{MF}}\xspace} 
   
    \newcommand{\QT}{\ensuremath{\mathcal{QT}}\xspace} 

\newcommand{\mcg}{\ensuremath{\mathcal{MCG}}\xspace}

\newcommand{\mf}{\ensuremath{\mathcal{MF}}\xspace} 
  \newcommand{\pmf}{\ensuremath{\mathcal{PMF}}\xspace} 
   
\newcommand{\qt}{\ensuremath{\mathcal{QT}}\xspace} 
       
    \newcommand{\diam}{\operatorname{diam}}

  \newcommand{\Teich}{Teich\-m\"u\-ller\ } 
  
  \newcommand{\pa}{pseudo-Anosov}

  \newcommand{\bpsi}{{\overline \psi}}

  \newcommand{\param}{{\mathchoice{\mkern1mu\mbox{\raise2.2pt\hbox{$
  \centerdot$}}
  \mkern1mu}{\mkern1mu\mbox{\raise2.2pt\hbox{$\centerdot$}}\mkern1mu}{
  \mkern1.5mu\centerdot\mkern1.5mu}{\mkern1.5mu\centerdot\mkern1.5mu}}}

  \renewcommand{\setminus}{{\smallsetminus}}

  \newcommand{\from}{\colon\thinspace}